\documentclass[a4paper,twoside]{article}
\usepackage{a4}
\usepackage{amssymb}
\usepackage{amsmath}
\usepackage{upref} 
\usepackage[colorlinks,citecolor=blue,linkcolor=blue]{hyperref}
\usepackage[dvipsnames]{color}
\usepackage[active]{srcltx}
\allowdisplaybreaks[2] 
\newcount\minutes \newcount\hours
\hours=\time
\divide\hours 60
\minutes=\hours
\multiply\minutes -60
\advance\minutes \time
\newcommand{\klockan}{\the\hours:{\ifnum\minutes<10 0\fi}\the\minutes}
\newcommand{\tid}{\today\ \klockan}
\newcommand{\prtid}{\smash{\raise 10mm \hbox{\LaTeX ed \tid}}}
\renewcommand{\prtid}{}
\makeatletter
\def\sectionmark#1{} 
\def\subsectionmark#1{}
\newcommand{\sectnr}{\ifnum \c@secnumdepth >\z@
                 \thesection.\hskip 1em\relax \fi}
\def\@evenhead{\footnotesize\rm\thepage\hfil\leftmark\hfil\llap{\prtid}}
\def\@oddhead{\footnotesize\rm\rlap{\prtid}\hfil\rightmark\hfil\thepage}
\def\tableofcontents{\section*{Contents} 
 \@starttoc{toc}}
\makeatother
\makeatletter
\def\@biblabel#1{#1.}
\makeatother
\makeatletter
\let\Thebibliography=\thebibliography
\renewcommand{\thebibliography}[1]{\def\@mkboth##1##2{}\Thebibliography{#1}
\addcontentsline{toc}{section}{References}
\frenchspacing 
\setlength{\@topsep}{0pt}
\setlength{\itemsep}{0pt}%
\setlength{\parskip}{0pt plus 2pt}%
}
\makeatother
\makeatletter
\def\mdots@{\mathinner.\nonscript\!.%
 \ifx\next,.\else\ifx\next;.\else\ifx\next..\else
 \nonscript\!\mathinner.\fi\fi\fi}
\let\ldots\mdots@
\let\cdots\mdots@
\let\dotso\mdots@
\let\dotsb\mdots@
\let\dotsm\mdots@
\let\dotsc\mdots@
\def\vdots{\vbox{\baselineskip2.8\p@ \lineskiplimit\z@
    \kern6\p@\hbox{.}\hbox{.}\hbox{.}\kern3\p@}}
\def\ddots{\mathinner{\mkern1mu\raise8.6\p@\vbox{\kern7\p@\hbox{.}}%
    \raise5.8\p@\hbox{.}\raise3\p@\hbox{.}\mkern1mu}}
\makeatother
\makeatletter
\def\@seccntformat#1{\csname the#1\endcsname.\quad}
\makeatother
\makeatletter
\long\def\@makecaption#1#2{%
  \vskip\abovecaptionskip
  \sbox\@tempboxa{ #1. #2}%
  \ifdim \wd\@tempboxa >\hsize
    #1. #2\par
  \else
    \global \@minipagefalse
    \hb@xt@\hsize{\hfil\box\@tempboxa\hfil}%
  \fi
  \vskip\belowcaptionskip}
\makeatother
\newcommand{\authortitle}[2]{\author{#1}\title{#2}\markboth{#1}{#2}}
\newcommand{\art}[6]{{\sc #1, \rm #2, \it #3 \bf #4 \rm (#5), \mbox{#6}.}}
\newcommand{\artprep}[3]{{\sc #1, \rm #2, #3.}}
\newcommand{\artin}[3]{{\sc #1, \rm #2,  in #3.}}
\newcommand{\book}[3]{{\sc #1, \it #2, \rm #3.}}
\newcommand{\auth}[2]{{#1, #2.}}
\newcommand{\AND}{{\rm and }}
\RequirePackage{amsthm}
\newtheoremstyle{descriptive}%
  {\topsep}   
  {\topsep}   
  {\rmfamily} 
  {}          
  {\bfseries} 
  {.}         
  { }         
  {}          
\newtheoremstyle{propositional}%
  {\topsep}   
  {\topsep}   
  {\itshape}  
  {}          
  {\bfseries} 
  {.}         
  { }         
  {}          
\theoremstyle{propositional}
\newtheorem{thm}{Theorem}[section]
\newtheorem{prop}[thm]{Proposition}
\newtheorem{lem}[thm]{Lemma}
\newtheorem{cor}[thm]{Corollary}
\theoremstyle{descriptive}
\newtheorem{deff}[thm]{Definition}
\newtheorem{example}[thm]{Example}
\newtheorem{remark}[thm]{Remark}
\makeatletter
\renewenvironment{proof}[1][\proofname]{\par
  \pushQED{\qed}%
  \normalfont
  \trivlist
  \item[\hskip\labelsep
        \itshape
    #1\@addpunct{.}]\ignorespaces
}{%
  \popQED\endtrivlist\@endpefalse
}
\makeatother
\newcommand{\setm}{\setminus}
\renewcommand{\subsetneq}{\varsubsetneq}
\renewcommand{\emptyset}{\varnothing}
\newcommand{\grad}{\nabla}
\DeclareMathOperator{\BMO}{BMO}
\DeclareMathOperator*{\essliminf}{ess\,lim\,inf}
\newcommand{\bdry}{\partial}
\newcommand{\bdy}{\bdry}
\newcommand{\loc}{_{\rm loc}}
{\catcode`p =12 \catcode`t =12 \gdef\eeaa#1pt{#1}}      
\def\accentadjtext#1{\setbox0\hbox{$#1$}\kern   
                \expandafter\eeaa\the\fontdimen1\textfont1 \ht0 }
\def\accentadjscript#1{\setbox0\hbox{$#1$}\kern 
                \expandafter\eeaa\the\fontdimen1\scriptfont1 \ht0 }
\def\accentadjscriptscript#1{\setbox0\hbox{$#1$}\kern   
                \expandafter\eeaa\the\fontdimen1\scriptscriptfont1 \ht0 }
\def\accentadjtextback#1{\setbox0\hbox{$#1$}\kern       
                -\expandafter\eeaa\the\fontdimen1\textfont1 \ht0 }
\def\accentadjscriptback#1{\setbox0\hbox{$#1$}\kern     
                -\expandafter\eeaa\the\fontdimen1\scriptfont1 \ht0 }
\def\accentadjscriptscriptback#1{\setbox0\hbox{$#1$}\kern 
                -\expandafter\eeaa\the\fontdimen1\scriptscriptfont1 \ht0 }
\def\itoverline#1{{\mathsurround0pt\mathchoice
        {\rlap{$\accentadjtext{\displaystyle #1}
                \accentadjtext{\vrule height1.593pt}
                \overline{\phantom{\displaystyle #1}
                \accentadjtextback{\displaystyle #1}}$}{#1}}
        {\rlap{$\accentadjtext{\textstyle #1}
                \accentadjtext{\vrule height1.593pt}
                \overline{\phantom{\textstyle #1}
                \accentadjtextback{\textstyle #1}}$}{#1}}
        {\rlap{$\accentadjscript{\scriptstyle #1}
                \accentadjscript{\vrule height1.593pt}
                \overline{\phantom{\scriptstyle #1}
                \accentadjscriptback{\scriptstyle #1}}$}{#1}}
        {\rlap{$\accentadjscriptscript{\scriptscriptstyle #1}
                \accentadjscriptscript{\vrule height1.593pt}
                \overline{\phantom{\scriptscriptstyle #1}
                \accentadjscriptscriptback{\scriptscriptstyle #1}}$}{#1}}}}
\newcommand{\limplus}{{\mathchoice{\vcenter{\hbox{$\scriptstyle +$}}}
  {\vcenter{\hbox{$\scriptstyle +$}}}
  {\vcenter{\hbox{$\scriptscriptstyle +$}}}
  {\vcenter{\hbox{$\scriptscriptstyle +$}}}
}}
\newcommand{\alp}{\alpha}
\newcommand{\al}{\alpha}
\newcommand{\be}{\beta}
\newcommand{\Om}{\Omega}
\renewcommand{\phi}{\varphi}
\newcommand{\tphi}{\widetilde{\phi}}
\newcommand{\eps}{\varepsilon}
\newcommand{\de}{\delta}
\newcommand{\ga}{\gamma}
\newcommand{\p}{{$p\mspace{1mu}$}}
\newcommand{\R}{\mathbf{R}}
\newcommand{\ub}{\bar{u}}
\newcommand{\tQ}{\widetilde{Q}}
\newcommand{\bQ}{\itoverline{Q}}

\newcommand{\Wp}{W^{1,p}}
\newcommand{\Wploc}{W^{1,p}\loc}
\newcommand{\Ei}{$(E_i)$}
\newcommand{\Ej}{$(E_j)$}
\newcommand{\Ek}{$(E_k)$}
\newcommand{\Eij}{$(E_{ij})$}
\newcommand{\Eik}{$(E_{ik})$}
\newcommand{\Ejk}{$(E_{jk})$}
\newcommand{\Eji}{$(E_{ji})$}
\newcommand{\Eki}{$(E_{ki})$}
\newcommand{\Ekj}{$(E_{kj})$}
\newcommand{\Ehi}{$(\widehat{E}_i)$}
\newcommand{\Ehj}{$(\widehat{E}_j)$}
\newcommand{\Ehk}{$(\widehat{E}_k)$}
\newcommand{\xhi}{\hat{x}_i}

\numberwithin{equation}{section}
\newenvironment{ack}{\medskip{\it Acknowledgement.}}{}

\begin{document}

\authortitle{Anders Bj\"orn, Jana Bj\"orn
and Riikka Korte}
{Minima of quasisuperminimizers}
\author{
Anders Bj\"orn \\
\it\small Department of Mathematics, Link\"oping University,\\
\it\small SE-581 83 Link\"oping, Sweden\/{\rm ;}
\it \small anders.bjorn@liu.se
\\
\\
Jana Bj\"orn \\
\it\small Department of Mathematics, Link\"oping University, \\
\it\small SE-581 83 Link\"oping, Sweden\/{\rm ;}
\it \small jana.bjorn@liu.se
\\
\\
Riikka Korte \\
\it\small Department of Mathematics, P.O. Box 68\/ 
\textup{(}Gustaf H\"allstr\"omin katu 2b\/\textup{)}, \\
\it\small FI-00014 University of Helsinki, Finland\/{\rm ;}
\it \small riikka.korte@helsinki.fi
}

\date{}
\maketitle

\noindent{\small {\bf Abstract}.  
Let $u_i$ be a $Q_i$-quasisuperminimizer, $i=1,2$, and $u=\min\{u_1,u_2\}$,
where $1 \le Q_1 \le Q_2$.
Then $u$ is a quasisuperminimizer, and we improve
upon the known upper bound (due to Kinnunen and Martio)
for the optimal quasisuperminimizing constant $Q$ of $u$.
We give the first examples with
$Q>Q_2$, and show that in general $Q>Q_2$ whenever $Q_1 >1$.
We also study the blowup of the quasisuperminimizing constant in pasting lemmas.
}



\section{Introduction}

Let $1<p<\infty$ and let $\Om \subset \R^n$ be a nonempty open set.
A function $u \in \Wploc(\Om)$ is a \emph{Q-quasiminimizer}, $Q \ge 1$,
 in $\Om$ if
\begin{equation} \label{eq-qm-Rn}
      \int_{\phi \ne 0} |\grad u|^p\, dx
	\le   Q    \int_{\phi \ne 0} |\grad (u+\phi)|^p\, dx
\end{equation}
for all $\phi \in \Wp_0(\Om)$.
A function $u$ is a \emph{$Q$-quasisuper\/\textup{(}sub\/\textup{)}minimizer} 
if \eqref{eq-qm-Rn}
 holds for all nonnegative (nonpositive) $\varphi\in W^{1,p}_0(\Omega)$.

Quasiminimizers were introduced by
Giaquinta and Giusti~\cite{GG1}, \cite{GG2} as a tool for a unified
treatment of variational integrals, elliptic equations and
quasiregular mappings on $\R^n$.  They realized that De Giorgi's
method could be extended to quasiminimizers, obtaining, in particular,
local H\"older continuity.  DiBenedetto and Trudinger~\cite{DiBTru}
proved the Harnack inequality for quasiminimizers, as well as weak
Harnack inequalities for quasisub- and quasisuperminimizers.  
A little later, 
Ziemer~\cite{Ziem86} gave a Wiener-type criterion sufficient for
boundary regularity for quasiminimizers,
and 
Tolksdorf~\cite{tolksdorf} obtained a Caccioppoli inequality and
a convexity result for quasiminimizers.
The results in \cite{DiBTru}--\cite{GG2}
and
\cite{Ziem86} were extended to metric spaces
by Kinnunen--Shanmugalingam~\cite{KiSh01} and J.~Bj\"orn~\cite{Bj02}
in the beginning of this century,
see also A.~Bj\"orn--Marola~\cite{BMarola}.
Soon afterwards, Kinnunen--Martio~\cite{KiMa03} showed
that quasiminimizers have an interesting potential theory,
in particular they introduced quasisuperharmonic functions, which
are related to quasisuperminimizers in a similar way as
superharmonic functions are related to supersolutions.
The theory of quasi\-(super)\-minimizers has been further studied
in 
\cite{ABkellogg}--\cite{BBpower}, \cite{BBM}, \cite{BMartio}, 
\cite{JBCalcVar}-- \cite{DiBen-Gia},
\cite{Ivert},
\cite{KiKoLa},
\cite{KiMaMa}, 
\cite{KoMaSh}--\cite{moscariello} and 
\cite{uppman}.

It is well known that the minimum 
of two superharmonic functions is again superharmonic.
This property is used extensively e.g.\ in balayage and in
the Perron method for solving the Dirichlet problem.
For quasisuperminimizers, 
Kinnunen--Martio~\cite{KiMa03} showed the following similar result.
(We formulate it in $\R^n$, but it is valid also in metric measure
spaces, see Section~\ref{sect-2}.
The same holds for Theorems~\ref{thm:min2} and~\ref{thm-paste-supermin}.)

\begin{thm}  \label{thm-KiMa}
\textup{(Kinnunen--Martio~\cite{KiMa03})}
Let $u_j$ be a $Q_j$-quasisuperminimizer, $j=1,2$.
Then $\min\{u_1,u_2\}$ is a $\min\{Q_1Q_2,Q_1+Q_2\}$-quasisuperminimizer.
\end{thm}

The blowup of the quasisuperminimizing constant in this result
is the main focus of this paper.
Our first result is the following better upper bound.

\begin{thm}\label{thm:min2}
Let $u_i$ be a $Q_i$-quasisuperminimizer in $\Omega$ for $i=1,2$. 
Then $u=\min\{u_1,u_2\}$ is a $\bQ$-quasisuperminimizer in $\Omega$, where
 \begin{equation}   \label{eq-def-q-for-min}
\bQ=\begin{cases}
  1, & \text{if } Q_1=Q_2=1, \\
  \displaystyle (Q_{1}+Q_{2}-2)\frac{Q_{1}Q_{2}}{Q_{1}Q_{2}-1}, & \text{otherwise.}
  \end{cases}
 \end{equation}
In particular, if $Q_1=Q_2$, then $\bQ=2Q_1^2/(Q_1+1)$.
\end{thm}

Note that when $Q_{1},Q_{2}>1$, we always have the following bounds for $\bQ$
in \eqref{eq-def-q-for-min}:
 \[
 Q_{1}+Q_{2}-2< \bQ<Q_{1}+Q_{2}-1< \min\{Q_{1}Q_{2},Q_{1}+Q_{2}\}.
 \]
This means that we obtain a better blowup constant than 
Kinnunen--Martio~\cite{KiMa03} whenever $Q_{1},Q_{2}>1$.

In the converse direction it is clear that $u$
cannot (in general) have a better quasisuperminimizing constant
than $\max\{Q_1,Q_2\}$
(and thus already Theorem~\ref{thm-KiMa} is optimal if $Q_1=1$ or $Q_2=1$).
As far as we know, there have so far not been any examples showing that
some blowup is indeed possible.
We construct such examples in Section~\ref{sect-lower}.
In particular, we prove the following result.

\begin{thm}   \label{thm-ex-blowup}
Let  $p>1$ and $1<Q_1\le Q_2$. Then there exist functions $u_1$ and $u_2$ on 
$(0,1)\subset\R$ such that $u_j$ is a $Q_j$-quasisuperminimizer in $(0,1)$, 
$j=1,2$, but $\min\{u_1,u_2\}$ is not a $Q_2$-quasisuperminimizer in $(0,1)$.
\end{thm}

We also obtain estimates for the blowup in the quasisuperminimizing constant.
In Section~\ref{sect-3-or-more} we give an upper bound for the blowup
when taking a minimum of three quasisuperminimizers, which is better
than iterating Theorem~\ref{thm:min2}.

Another result with a blowup in the quasisuperminimizing constant is the 
following pasting lemma.

\begin{thm} \label{thm-paste-supermin} 
\textup{(Bj\"orn--Martio~\cite[Theorem~4.1]{BMartio})}
Assume that\/ $\Om_1 \subset \Om_2 \subset \R^n$ are open and
that $u_j$ is a $Q_j$-quasisuperminimizer in\/ $\Om_j$, $j=1,2$.
Let
\[
     u=\begin{cases}
        u_2, & \text{in } \Om_2 \setm \Om_1, \\ 
        \min\{u_1,u_2\}, & \text{in } \Om_1.
	\end{cases}
\]
If $u \in \Wploc(\Om_2)$,
then $u$ 
is a $Q_1 Q_2$-quasisuperminimizer in\/ $\Om_2$.
\end{thm}

In Theorems~\ref{thm-Q1Q2-sharp} and~\ref{thm-paste-sharp} 
we show that the blowup
constant $Q_1 Q_2$ is optimal in this result.
There is also a similar pasting lemma for quasisuperharmonic
functions in Bj\"orn--Martio~\cite[Theorem~5.1]{BMartio} 
and our optimality result applies also to this case, see Remark~\ref{rmk-qsh}.

Yet another result with a blowup of the quasisuperminimizing constant
is the reflection
principle by Martio~\cite[Theorem~3.1]{martioReflect}. 
In one dimension (i.e.\ on $\R$) he obtained a better result
in Theorem~4.1 in \cite{martioReflect}.
The blowup constant in the latter result was subsequently
improved upon by Uppman~\cite[Lemma~2.8]{uppman},  who also showed
that his constant is the best possible.

\begin{ack}
The  first two authors were supported by the Swedish Research Council.
The third author was supported by the Academy of Finland, grant no.\ 250403.
Part of this research was done while 
the third author visited Link\"oping University in 2009, 
and while 
all three authors
visited Institut Mittag-Leffler in the autumn of 2013.
They want to thank the institute for the hospitality.
\end{ack}

\section{An upper bound for the blowup}
\label{sect-2}

In this section we are going to  prove Theorem~\ref{thm:min2}.
Let us however first discuss some consequences and generalizations
of it.

\begin{deff} A function $u : \Om \to (-\infty,\infty]$
is  \emph{$Q$-quasisuperharmonic} in $\Om$
if $u$ is not identically $\infty$ in any component of $\Om$,
$\min\{u,k\}$ is a $Q$-quasisuperminimizer in $\Om$ for
every $k \in \R$, and 
$u$ is  \emph{lower semicontinuously regularized}, i.e.\ 
\[
  u(x)=\essliminf_{y \to x} u(y) 
       \quad \text{for } x \in \Om.
\]
\end{deff}

This definition is equivalent to Definition~7.1 in Kinnunen--Martio~\cite{KiMa03},
see Theorem~7.10 in~\cite{KiMa03}.
Using this definition we obtain
the following corollary of Theorem~\ref{thm:min2}.

\begin{cor} \label{cor-min2}
Let $u_i$ be a $Q_i$-quasisuperharmonic function  in $\Omega$ for $i=1,2$. Then $u=\min\{u_1,u_2\}$ is
   $\bQ$-quasisuperharmonic in $\Omega$, 
where $\bQ$ is given by \eqref{eq-def-q-for-min}.
\end{cor}

We have formulated Theorem~\ref{thm:min2} and Corollary~\ref{cor-min2}
on (unweighted) $\R^n$, but they have direct counterparts
valid in complete metric spaces 
equipped with doubling measures supporting a \p-Poincar\'e
inequality (and thus also on weighted $\R^n$ with a \p-admissible weight),
see Bj\"orn--Bj\"orn~\cite{BBbook} for more on the metric space theory
(note that Appendix~C therein gives a short survey on quasiminimizers).

Below we have chosen to give an $\R^n$ proof of Theorem~\ref{thm:min2}.
However, it carries over verbatim to metric spaces,
with the trivial modifications that $|\nabla u|$ is replaced by the minimal
\p-weak upper gradient $g_u$ (and similarly for the other gradients)
and $dx$ is replaced by $d\mu$.
Note that $g_u=|\nabla u|$ on unweighted and weighted $\R^n$,
see Appendices~A.1 and~A.2 in \cite{BBbook}.

\begin{proof}[Proof of Theorem~\ref{thm:min2}]
Let $0\le\phi\in W^{1,p}_{0}(\Om)$ be arbitrary and set 
\begin{align*}
A &=\{x\in\Om:\phi(x)>0\}, \\
A_1 &=\{x \in A : u_1(x)<u_2(x)\},\\
A_2 &=\{x \in A : u_2(x)<u_1(x)\}, \\
A_0 &=\{x \in A : v(x)>\max\{u_1(x),u_2(x)\}\},
\end{align*}
where $v=u+\phi$.
Note that $A=A_1\cup A_2\cup A_0$, though not pairwise disjointly.

We may assume that $\int_A|\grad u|^p\,dx<\infty$, as otherwise~\eqref{eq-qm-Rn}
holds trivially, since the triangle inequality together with the fact that
$\phi\in W^{1,p}(\Om)$ implies that
\[
\biggl( \int_A|\grad (u+\phi)|^p\,dx \biggr)^{1/p}
\ge \biggl( \int_A|\grad u|^p\,dx \biggr)^{1/p}
          - \biggl( \int_A|\grad\phi|^p\,dx \biggr)^{1/p} = \infty.
\]
Let $\phi_1=(\min\{u_2,v\}-u_1)_\limplus$ and note that
$0\le\phi_1\le\phi$, which implies that $\phi_1\in W^{1,p}_{0}(\Om)$.
The $Q_1$-quasisuperminimizing property of $u_1$ yields
\begin{equation}\label{eq-u1-on-A1}
\int_{\phi_1>0}|\grad u_1|^p\,dx
\le Q_1 \int_{\phi_1>0}|\grad (u_1+\phi_1)|^p\,dx. 
\end{equation}
Note that $\phi_1(x)>0$ if and only if $u_2(x)>u_1(x)$ and 
$u_1(x)+\phi(x)=v(x)>u_1(x)$,
which in turn holds exactly when $x\in A_1$.
Moreover, 
\[
u_1+\phi_1= \begin{cases}
        u_2, & \text{in } A_1\cap A_0, \\ 
        v, & \text{in } A_1\setm A_0.
	\end{cases}
\]
Multiplying~\eqref{eq-u1-on-A1} by $(Q_2-1)$ then gives
\begin{equation}\label{eq-A1-v}
(Q_2-1) \int_{A_1}|\grad u_1|^p\,dx
\le Q_1(Q_2-1) \biggl( \int_{A_1\cap A_0}|\grad u_2|^p\,dx 
         + \int_{A_1\setm A_0}|\grad v|^p\,dx \biggr). 
\end{equation}
Similarly, using $\phi_2=(\min\{u_1,v\}-u_2)_\limplus\in W^{1,p}_{0}(\Om)$ and the
$Q_2$-quasisuperminimizing property of $u_2$ we obtain (after multiplication
with $(Q_1-1)$),
\begin{equation}\label{eq-A2-v}
(Q_1-1) \int_{A_2}|\grad u_2|^p\,dx
\le Q_2(Q_1-1) \biggl( \int_{A_2\cap A_0}|\grad u_1|^p\,dx 
         + \int_{A_2\setm A_0}|\grad v|^p\,dx \biggr). 
\end{equation}

Next, let $\tphi_j=(v-u_j)_\limplus$, $j=1,2$.
Since $0\le \tphi_j\le\phi$, we have $\tphi_j\in W^{1,p}_{0}(\Om)$ 
and~\eqref{eq-qm-Rn} with $u_j$ and $\tphi_j$ gives
\begin{equation}  \label{eq-test-tphi-j}
\int_{\tphi_j>0}|\grad u_j|^p\,dx
\le Q_j \int_{\tphi_j>0}|\grad (u_j+\tphi_j)|^p\,dx. 
\end{equation}
Now, $\tphi_1(x)>0$ if and only if $\min\{u_1(x),u_2(x)\}+\phi(x)>u_1(x)$,
which is equivalent to $x\in A$ (i.e.\ $\phi(x)>0$)
and $u_2(x)+\phi(x)>u_1(x)$.
This in turn holds exactly if $x\in A_1$ or 
$u_2(x)\le u_1(x)<u_2(x)+\phi(x)=v(x)$, i.e.\ when $x\in A_1\cup A_0$.
Similarly, $\tphi_2(x)>0$ if and only if $x\in A_2\cup A_0$.
 
Since $u_j+\tphi_j=v$ whenever $\tphi_j>0$, the inequalities 
in~\eqref{eq-test-tphi-j} give
\begin{equation}   \label{eq-est-u1-v}
Q_2(Q_1-1) \int_{A_1\cup A_0} |\grad u_1|^p\,dx
\le Q_1 Q_2(Q_1-1) \int_{A_1\cup A_0} |\grad v|^p\,dx
\end{equation}
and 
\begin{equation}   \label{eq-est-u2-v}
Q_1(Q_2-1) \int_{A_2\cup A_0} |\grad u_2|^p\,dx
\le Q_1 Q_2(Q_2-1) \int_{A_2\cup A_0} |\grad v|^p\,dx,
\end{equation}
where we have also multiplied by $Q_2(Q_1-1)$ and $Q_1(Q_2-1)$, respectively.

Next, we shall sum up the inequalities \eqref{eq-A1-v}, \eqref{eq-A2-v},
\eqref{eq-est-u1-v} and \eqref{eq-est-u2-v} as follows.
The first term in the right-hand side of~\eqref{eq-A1-v} can be subtracted
from the left-hand side of~\eqref{eq-est-u2-v}, leaving 
\begin{align*}
&Q_1(Q_2-1) \int_{(A_2\cup A_0)\setm A_1} |\grad u_2|^p\,dx \\
&\kern 8em
= Q_1(Q_2-1) \biggl( \int_{A_2} |\grad u_2|^p\,dx
+ \int_{A_0\setm(A_1\cup A_2)} |\grad u_2|^p\,dx \biggr)
\end{align*}
therein.
Since $u=u_2$ in $A\setm A_1\supset A_2$, 
adding this to the left-hand side of~\eqref{eq-A2-v} results in
\begin{align}  \label{eq-contrib-LHS-A1}
&(Q_1(Q_2-1) + (Q_1-1)) \int_{A_2} |\grad u|^p\,dx
+ Q_1(Q_2-1) \int_{A_0\setm(A_1\cup A_2)} |\grad u|^p\,dx \nonumber\\
&\kern 5em
= (Q_1Q_2-1) \int_{A_2} |\grad u|^p\,dx
+ Q_1(Q_2-1) \int_{A_0\setm(A_1\cup A_2)} |\grad u|^p\,dx
\end{align}
as $|\grad u_2|$'s contribution to the left-hand side of the final sum.

Similarly, subtracting the first term in the right-hand side 
of~\eqref{eq-A2-v} from the left-hand side of~\eqref{eq-est-u1-v}, 
and adding the left-hand side of~\eqref{eq-A1-v} contributes with 
\begin{equation}  \label{eq-contrib-LHS-A2}
(Q_1Q_2-1) \int_{A_1} |\grad u|^p\,dx
+ Q_2(Q_1-1) \int_{A_0\setm(A_1\cup A_2)} |\grad u|^p\,dx
\end{equation}
to the left-hand side of the final sum.
Since $Q_1(Q_2-1)+Q_2(Q_1-1) \ge Q_1Q_2-1$, summing up~\eqref{eq-contrib-LHS-A1}
and~\eqref{eq-contrib-LHS-A2} shows that the left-hand side in the final sum
will be 
\begin{align*}
&(Q_1Q_2-1) \int_{A_1\cup A_2} |\grad u|^p\,dx
+ (Q_1(Q_2-1)+Q_2(Q_1-1)) \int_{A_0\setm(A_1\cup A_2)} |\grad u|^p\,dx \\
&\kern 15em 
\ge (Q_1Q_2-1) \int_{A_1\cup A_2\cup A_0} |\grad u|^p\,dx.
\end{align*}

We now turn to the right-hand side of the sum of 
\eqref{eq-A1-v}, \eqref{eq-A2-v}, \eqref{eq-est-u1-v} and \eqref{eq-est-u2-v}.
The remaining term in the right-hand side of~\eqref{eq-A1-v} is
\[
Q_1(Q_2-1) \int_{A_1\setm A_0}|\grad v|^p\,dx 
\le Q_1Q_2(Q_2-1) \int_{A_1\setm A_0}|\grad v|^p\,dx, 
\]
which together with the right-hand side of~\eqref{eq-est-u2-v} contributes
with
\[
Q_1Q_2(Q_2-1) \int_{A_1\cup A_2\cup A_0}|\grad v|^p\,dx
\]
to the right-hand side of the final sum.
Similarly, the remaining term in the right-hand side of~\eqref{eq-A2-v}
together with the right-hand side of~\eqref{eq-est-u1-v} gives
\[
Q_1Q_2(Q_1-1) \int_{A_1\cup A_2\cup A_0}|\grad v|^p\,dx
\]
in the right-hand side of the final sum.

As $A_1\cup A_2\cup A_0=A=\{x\in\Om:\phi(x)>0\}$, we have thus obtained
\[
(Q_1Q_2-1) \int_{\phi>0} |\grad u|^p\,dx
\le Q_1Q_2(Q_1+Q_2-2)\int_{\phi>0}|\nabla v|^p\,dx.
\]
Division by $Q_1Q_2-1$ concludes the proof of Theorem~\ref{thm:min2}.
(If $Q_1=Q_2=1$, the result follows from Theorem~\ref{thm-KiMa}.)
\end{proof}

\section{Lower bounds for the blowup}
\label{sect-lower}

Consider two quasisuperminimizers defined on some open set $\Om$.
More precisely let $u_j$ be a $Q_j$-quasisuperminimizer in $\Om$, $j=1,2$.
Also let $u=\min\{u_1,u_2\}$ and assume that $Q_1 \le Q_2$.
Theorem~\ref{thm-KiMa} then shows that $u$ is also a quasisuperminimizer,
and it gives an upper bound on the optimal quasisuperminimizer constant $Q$ 
for
$u$ (in terms of $Q_1$ and $Q_2$ only).
In Theorem~\ref{thm:min2} we improved upon this upper bound.

As far as we know, 
there have not been any examples showing that the optimal
$Q$ can be greater than $Q_2$. 
It is obvious that one cannot do any better than $Q_2$ 
in general (just consider the cases when $u_1 \ge u_2$ in $\Om$).
Note also that if $Q_1=1$, then Theorem~\ref{thm-KiMa}
shows that $u$ is a $Q_2$-quasisuperminimizer, and hence
the constants in Theorems~\ref{thm-KiMa} and~\ref{thm:min2} 
are sharp in this case.

In this section we will give several examples of pairs
of quasisuperminimizers such that their minimum
has a blowup of the quasisuperminimizer constant,
i.e.\ in the notation above we get $Q > Q_2$.
Even though the best (largest) bounds come just from one such example
we feel that it can be of interest to mention several different examples
as they may add a little to the knowledge on quasisuperminimizers.

Let us already now mention that in all our examples, the functions
$u_1$ and $u_2$ will not only be quasisuperminimizers,
but will in fact be quasiminimizers (with the same optimal constants)
as well as subminimizers
(i.e.\ $1$-quasisubminimizers).

We will also prove Theorem~\ref{thm-ex-blowup}, i.e.\ 
that whenever $Q_1 >1$, then  there are examples
showing that one can have $Q > Q_2$ and thus that 
$\max\{Q_1,Q_2\}$ is an upper bound only  when $Q_1=1$.

Our examples will all be on $\R$.
The reason for this is that this is almost the only case
when one can actually calculate optimal quasiminimizers and their constants.
As far as we know, the only higher-dimensional quasi(super)minimizers
for which their optimal quasi(super)minimizer constant has been determined,
and is strictly larger than $1$, are the power-type
quasi(super)minimizers studied in 
Bj\"orn--Bj\"orn~\cite{BBpower}.

The easiest example of a blowup in the quasisuperminimizing constant
is perhaps the following.
(It was incidentally also the first example we discovered.)

\begin{example} \label{ex-1}
Let $p=2$, 
\[
    u_1(x)=\begin{cases}
        \tfrac{2}{3}x, & 0 \le x \le \tfrac{1}{2}, \\
        \tfrac{4}{3}x-\tfrac{1}{3}, & \tfrac{1}{2} \le x \le 1, 
	\end{cases}
\quad \text{and} \quad
    u_2(x)=\begin{cases}
        \tfrac{5}{6}x, & 0 \le x \le \tfrac{4}{5}, \\
        \tfrac{5}{3}x-\tfrac{2}{3}, & \tfrac{4}{5} \le x \le 1. 
	\end{cases}
\]
Then $u_1$ and $u_2$ are $\frac{9}{8}$-quasisuperminimizers 
(with $\frac{9}{8}=1.125$ being the optimal constant), 
by Theorem~\ref{thm-ga-p} below.
We will call functions such as $u_1$ and $u_2$ \emph{one-corner functions}.

Let $u=\min\{u_1,u_2\}$.
Note that $u_1(x)=u_2(x)$ for $x=0,\tfrac{2}{3},1$.
Then
\[
    \int_0^1 (u')^2 \, dx= \tfrac{1}{2} \bigl(\tfrac{2}{3}\bigr)^2
	       + \bigl(\tfrac{2}{3}- \tfrac{1}{2}) \bigl(\tfrac{4}{3}\bigr)^2
	       + \bigl(\tfrac{4}{5}- \tfrac{2}{3}) \bigl(\tfrac{5}{6}\bigr)^2
	       + \tfrac{1}{5} \bigl(\tfrac{5}{3}\bigr)^2
	    = \tfrac{7}{6}.
\]
Comparison with $v(x)=x$ shows that 
$u$ is not a $Q$-quasisuperminimizer for any $Q<\tfrac{7}{6}=1.1666\ldots$\,.
The upper bounds given by Theorems~\ref{thm-KiMa} and~\ref{thm:min2}
 are
\[
    \tfrac{81}{64} = 1.265625
	\quad \text{and} \quad
	\tfrac{81}{68} = 1.191176\ldots.
\]

With $Q_1=Q_2=\frac{9}{8}$ and $p=2$ this example has been optimized,
i.e.\ $u_1$ and $u_2$ are one-corner functions with quotient between the slopes
$\ga=2$ and the choices of their corner points have been optimized
to get as large blowup as possible. 
For $p=2$ it is a rather straightforward (although a bit lengthy) 
calculation to do this optimization
by hand even for a general $Q=Q_1=Q_2$, and it leads to the  lower bound
$\frac{4}{3}Q -\frac{1}{3}$. 
We omit the details as we find better lower bounds below.

For other values of $p$ such optimization becomes more laborious,
and we decided to do some such calculations using Maple 16.
Some obtained values, correctly rounded to the nearest digit,
are shown in Table~\ref{table-1}.
These calculations suggest that for a given $Q=Q_1=Q_2$ the lower bounds
increase with $p$, but the dependence on $p$ is very small
(much smaller than we had expected). 

\begin{table}[htb]   
\begin{center}
\begin{tabular}{|l|r|r|r|r|}
\hline
\multicolumn{1}{|c|}{$Q$} & \multicolumn{1}{c|}{$p=1.2$} & \multicolumn{1}{c|}{$p=2$} 
& \multicolumn{1}{c|}{$p=100$} & 
	Upper bound \\
&&&&$2Q^2/(Q+1)$  \\ \hline
1.001 & 1.001333193  & 1.001333333 & 1.001333353 & 1.001500250 \\ \hline
1.01 &  1.013319341  & 1.013333333 & 1.013335243 & 1.015024876  \\ \hline
1.125 & 1.164635987  & 1.166666667 & 1.166948556 & 1.191176471 \\ \hline 
2  &  2.254420532&2.333333333            &  2.346323188 & 2.666666667 \\ \hline
10 & 11.80468177&13\phantom{.00000000} & 13.34762304 &18.18181818 \\ \hline
100&118.9796468&133\phantom{.0000000}&139.1598599 &198.0198020 \\ \hline
\end{tabular}
\end{center}
\caption{}\label{table-1}
\end{table}
\end{example}

\begin{remark}
Even though Example~\ref{ex-1} has been optimized it should be possible, 
by considering more general piecewise linear functions
and optimizing their parameters, to obtain better results, possibly even 
reaching the optimal constant for fixed
$p$, $Q_1$ and $Q_2$, at least when the result is specialized to $\R$.
Here, Theorem~4.1 in Martio~\cite{martioLiU09}  (which can also be found
as Theorem~C.2 in \cite{BBbook}) might be of help since it makes it possible
to approximate quasiminimizers by other (e.g.\ piecewise linear) 
functions with almost the same quasiminimizing constant.

Another necessary ingredient would be a good control of the best
quasiminimizing constant of such functions.
Lemmas~2 and~8 in Martio--Sbordone~\cite{MaSb} show that the quasiminimizing
constant is at most $(\sup |u'|/\inf|u'|)^{p-1}$.
In particular, all strictly increasing continuous piecewise linear functions 
(with finitely many corners) are
quasiminimizers, but the best constant is not easy to obtain.
Our Proposition~\ref{prop-optimal-zig-zag} below is a partial 
step in that direction.
\end{remark}

The above considerations open up for further numerical investigations 
of the blow up. 
We will not pursue this route as the following approach gives good 
lower bounds.

\begin{deff} \label{deff-determining}
If $u$ is a $Q$-quasiminimizer in $\Om \subset \R$ we say that
$u$ has the \emph{maximal \p-energy allowed by $Q$}
on an interval $I\subset\Om$ if
\[
       \int_I |u'|^p \, dx = Q \int_I |v'|^p \, dx
= Q \frac{(u(b)-u(a))^p}{(b-a)^{p-1}},
\]
where $v$ is the minimizer in $I$ with boundary values
$v=u$ on $\bdy I$, i.e.\ 
\[
v(x)=u(a)+\frac{u(b)-u(a)}{b-a}(x-a),
\]
where $a<b$ are the end points of $I$.
\end{deff}

\begin{example}  \label{ex-power}
For $\al>1-1/p$ and $x\in[0,1]$ let $v_\al(x)=x^\al$.
Theorem~6.2 in Bj\"orn--Bj\"orn~\cite{BBpower} with $n=1$ and $p>1$ implies 
that $v_\al$ is a $Q_\al$-quasiminimizer in $(0,1)$, where
\begin{equation}  \label{eq-def-Q-al}
Q_\al=\frac{\al^p}{p(\al-1)+1}
\end{equation}
is optimal.
In fact, if $1-1/p<\al\le1$, then $v_\al$ is a superminimizer and a
$Q_\al$-quasisubminimizer, while for $\al\ge1$, $v_\al$ is a subminimizer and a
$Q_\al$-quasi\-super\-minimizer in $(0,1)$.

A simple calculation also shows that for every $x_0\in(0,1)$,
\[
\int_0^{x_0} (v_\al')^p\,dx 
= \int_0^{x_0} \al^p x^{p(\al-1)}\,dx 
= \frac{\al^p x_0^{p(\al-1)+1}}{p(\al-1)+1}
= Q_\al \int_0^{x_0} (x_0^{\al-1})^p \,dx,
\] 
where the latter integral is the \p-energy of the linear segment from the
origin to the point $(x_0,v_\al(x_0))$.
Thus, for every $x_0\in(0,1]$, $v_\al$ has the maximal \p-energy in $(0,x_0)$ 
allowed by $Q_\al$. 

Note that, given $Q>1$, there are exactly two exponents $1-1/p<\al'<1<\al$
such that $Q=Q_\al=Q_{\al'}$.
This is easily shown by differentiating~\eqref{eq-def-Q-al} and noting
that the derivative is negative for $\al<1$ and positive for $\al>1$,
and that $Q_\al\to\infty$ as $\al\to 1-1/p$ and as $\al\to\infty$.
We let 
\begin{equation}    \label{eq-def-uQ-ubQ}
u_Q(x)=x^\al \quad \text{and} \quad \ub_Q(x)=1-(1-x)^{\al'}.
\end{equation}
Then $u_Q(0)=\ub_Q(0)=0$ and $u_Q(1)=\ub_Q(1)=1$.
Note that both $u_Q$ and $\ub_Q$ are subminimizers and $Q$-quasisuperminimizers 
in $(0,1)$.
Moreover, $u_Q$ has the maximal \p-energy allowed by $Q$ on each interval 
$(0,x_0)$, while $\ub_Q$ has the maximal \p-energy 
allowed by $Q$ on each interval $(x_0,1)$.
\end{example}

We can now use the functions $u_Q$ and $\ub_Q$ above to 
prove Theorem~\ref{thm-ex-blowup}.

\begin{proof}[Proof of Theorem~\ref{thm-ex-blowup}.]
By Theorem~\ref{thm:min2}, the function $u:=\min\{u_{Q_1},\ub_{Q_2}\}$ is a
quasisuperminimizer in $(0,1)$ with a quasisuperminimizing constant 
given by~\eqref{eq-def-q-for-min}.
We shall show that $u$ is not a $Q_2$-quasisuperminimizer.
To do this, it suffices to show that the \p-energy
\[
\int_0^1 (u')^p\,dx>Q_2.
\]
Since $\ub_{Q_2}$ is a subminimizer in $(0,1)$ 
(by Theorem~6.2 in~\cite{BBpower}), we have that
\[
\int_0^{x_0} (u_{Q_1}')^p\,dx 
 > \int_0^{x_0} (\ub_{Q_2}')^p\,dx,
\]
where the strict inequality follows from the uniqueness of
solutions to obstacle problems (see e.g.\ Theorem~7.2 in~\cite{BBbook})
and from the fact that $u_{Q_1}<\ub_{Q_2}$ in a set of
positive measure. Hence
\begin{equation}   \label{eq-energy-ge-Q1}
\int_0^1 (u')^p\,dx 
> \int_0^1 (\ub_{Q_2}')^p\,dx=Q_2,
\end{equation}
which finishes the proof.
\end{proof}

Theorem~\ref{thm-ex-blowup} shows that in general there is a blow up in the 
quasisuperminimizing constant when taking minimum of two quasisuperminimizers,
but it does not give any quantitative estimate of the blow up.
Next, we shall give some lower bounds for the blow up.

Given $Q_1,Q_2>1$, let $1-1/p < \al_2<1 < \al_1$
be such that $Q_1=Q_{\al_1}$, $Q_2=Q_{\al_2}$ 
and $u_{Q_1}$ and $\ub_{Q_2}$ are the corresponding quasiminimizers.
Let $x_0$ be the unique number in $(0,1)$ such that $u_{Q_1}(x_0)=\ub_{Q_2}(x_0)$,
i.e.\ the unique solution of the equation 
\begin{equation} \label{eq-x0}
x_0^{\al_1} + (1-x_0)^{\al_2} =1.
\end{equation}
(To see that there is a unique solution, consider $w=\ub_{Q_2}-u_{Q_1}$
and note that $w(0)=w(1)=0$. 
Since $w'(0)>0$ and $w'(1)=\infty$,
there is at least one $x\in(0,1)$ such that $w(x)=0$.
Next, a simple calculation shows that $w'(x)=0$
if and only if 
\[
    v(x):= x^\be(1-x) =\biggl(\frac{\alp_1}{\alp_2}\biggr)^{1/(\alp_2-1)},
    \quad \text{where } \be = \frac{\alp_1-1}{1-\alp_2}>0.
\]
As $v(0)=v(1)=0$ and $v(x)$ attains its maximum  
at (and only at) $x=\be/(\be+1)$ we see that
there are at most two solutions to $w'(x)=0$,
and thus there can be at most one solution to \eqref{eq-x0},
which must lie  in between those two local extrema of $w$.

The \p-energy of $u=\min\{u_{Q_1},\ub_{Q_2}\}$ is then
(by comparing with the \p-energies of the linear
segments connecting the origin, the point $(x_0,x_0^{\al_1})$ and $(1,1)$)
\begin{equation}   \label{eq-def-tQ}
\int_0^1 (u')^p\,dx = Q_1 x_0^{p(\al_1-1)+1} + Q_2(1-x_0)^{p(\al_2-1)+1} =: \tQ.
\end{equation}
Here we have used that both $u_{Q_1}$ and $\ub_{Q_2}$ have the 
maximal energies allowed by $Q_1$ and $Q_2$ in the respective intervals.

Note that $x_0$ is uniquely determined by $Q_1$ and $Q_2$ (through $\al_1$ 
and $\al_2$) and thus $\tQ$ depends only on $Q_1$ and $Q_2$ (and on $p$).
Comparing this \p-energy with the \p-energy of the linear function $u_1(x)=x$
shows that $\tQ$ is a lower bound for the quasisuperminimizing constant of $u$.
We would therefore like to estimate $\tQ$.

The lower bounds in Table~\ref{table-2} have been obtained by letting Maple 16 
evaluate $\tQ$ for some values of $Q:=Q_1=Q_2$ and are compared with
the upper bound obtained in Theorem~\ref{thm:min2}.
Note that these lower bounds are considerably larger,
and much closer to the upper bounds, than those
in Table~\ref{table-1}.

\begin{table}[t]
\begin{center}
\begin{tabular}{|l|r|r|r|r|r|}
\hline
\multicolumn{1}{|c|}{$Q$} & \multicolumn{1}{c|}{$p=1.2$} & \multicolumn{1}{c|}{$p=2$} 
& \multicolumn{1}{c|}{$p=100$} & 
	Upper bound  & \multicolumn{1}{c|}{$\tQ^{\vphantom{k^{k^k}}}$ in~\eqref{Qt}}\\
&&&&\multicolumn{1}{c|}{$2Q^2/(Q+1)$} & \multicolumn{1}{c|}{for $p=2$}\\ \hline
1.001 &1.001480628 &1.001480660 &1.001480665 &1.001500250&1.001373803 \\ \hline
1.01  &1.014821935 &1.014825154 &1.014825593 &1.015024876&1.013873175  \\ \hline
1.125 &1.187625011 &1.188100103 &1.188165836 &1.191176471&1.180555556 \\ \hline 
2     &2.599606519 &2.619135721 &2.622161265 &2.666666667&2.601317394 \\ \hline
10    &17.45294063 &17.67321156 &17.72170691 &18.18181818&17.66438145 \\ \hline
100   &195.7168148 &196.3948537 &196.5955633 &198.0198020&196.3936712 \\ \hline
\end{tabular}
\end{center}
\caption{}\label{table-2}
\end{table}

Our next aim is to obtain more explicit estimates of $\tQ$. 
Calculating $\tQ$ in~\eqref{eq-def-tQ}
involves first solving the equation~\eqref{eq-def-Q-al}
twice for $\al$,
so that $Q_1=Q_{\al_1}$ and $Q_2=Q_{\al_2}$ as above, then finding
$0<x_0<1$ such that $x_0^{\al_1}+(1-x_0)^{\al_2}=1$, and finally evaluating
$\tQ$ for the obtained values of $\al_1$, $\al_2$ and $x_0$. 
This can be done numerically but not analytically 
(not even for $p=2$).

A somewhat weaker, but more explicit, estimate for $\tQ$ can be obtained 
in the following way.
Let $x_1\in(0,1)$ be such that $u_{Q_1}(x_1)=\al_2 x_1$, i.e.\ 
$x_1=\al_2^{1/(\al_1-1)}$.
Since $u_{Q_1}(0)=0$, $\ub'_{Q_2}(0)=\al_2$ and both $u_{Q_1}$ and $\ub_{Q_2}$ 
are convex, we have that
\[
\ub_{Q_2}(x)>\al_2 x>u_{Q_1}(x) \quad \text{for all } x\in(0,x_1). 
\]
In particular, $x_1<x_0$. 

As $\ub_{Q_2}$ is a subminimizer in $(0,1)$ and 
$\ub_{Q_2} > \max\{ u_{Q_1},\al_2 x\}$ in $(0,x_0)$, we then obtain 
(using also that $u_{Q_1}(x_0)=\ub_{Q_2}(x_0)$)
\begin{equation}  \label{eq-ub-Q2-submin-0-x0}
\int_0^{x_0} (\ub'_{Q_2})^p\,dx 
   < \int_0^{x_1} \al_2^p\,dx + \int_{x_1}^{x_0} (u'_{Q_1})^p\,dx,
\end{equation}
where the strict inequality follows as in~\eqref{eq-energy-ge-Q1}
from the uniqueness of solutions to obstacle problems.
From the fact that $u_{Q_1}$ has the maximal \p-energy allowed by $Q_1$ 
on the interval $(0,x_1)$ we can conclude that
\[
\int_0^{x_0} (u'_{Q_1})^p\,dx 
= \int_0^{x_1} (u'_{Q_1})^p\,dx + \int_{x_1}^{x_0} (u'_{Q_1})^p\,dx 
= Q_1\int_0^{x_1}\al_2^p\,dx + \int_{x_1}^{x_0} (u'_{Q_1})^p\,dx.
\]
Together with~\eqref{eq-ub-Q2-submin-0-x0} this yields
\begin{align}  \label{eq-blowup-1}
\tQ-Q_2 &= \int_0^1((u')^p-(\ub'_{Q_2})^p)\,dx 
= \int_0^{x_0} (u'_{Q_1})^p \,dx - \int_0^{x_0}(\ub'_{Q_2})^p\,dx \\
&> (Q_1-1) \int_0^{x_1} \al_2^p\,dx = (Q_1-1) \al_2^p x_1 
= (Q_1-1) \al_2^{p+1/(\al_1-1)}. \nonumber
\end{align}
(This gives another proof of Theorem~\ref{thm-ex-blowup}.)
A similar argument shows that
\begin{align}  \label{eq-blowup-2}
\tQ-Q_1 > (Q_2-1) \al_1^p (1-x_2)= (Q_2-1)\al_1^{p+1/(\al_2-1)},
\end{align}
where $x_2\in(0,1)$ is the solution of 
$\ub_{Q_2}(x_2)=1-\al_1(1-x_2)$, i.e.\ $1-x_2=\al_1^{1/(\al_2-1)}$.
Note that $x_2>x_0$.
Depending on the particular values of $p$, $Q_1$ and $Q_2$, one of
\eqref{eq-blowup-1} and~\eqref{eq-blowup-2} may be better than the
other.

For $p=2$, when $\al_1$ and $\al_2$ can be explicitly calculated in terms of 
$Q_1$ and $Q_2$, we get after simplification (and for $Q_1\le Q_2$) 
that the blow up is at least the maximum of
\begin{align}
\tQ-Q_2 &> (Q_1-1) \Bigl(Q_2+\sqrt{Q_2^2-Q_2} \Bigr)^{1-\sqrt{Q_1/(Q_1-1)}}, \label{Qt} \\
\tQ-Q_2 &> (Q_1-1) \Bigl(Q_2-\sqrt{Q_2^2-Q_2} \Bigr)^{1+\sqrt{Q_1/(Q_1-1)}}. \notag
\end{align}
For the values considered in Tables~\ref{table-1} and~\ref{table-2}, the first
estimate above is quite close to those in Table~\ref{table-2} and better
than those in Table~\ref{table-1}.

For $p\ne2$, we cannot obtain such explicit expressions. 
However, using
Remark~\ref{rmk-Qag} and \eqref{eq-x0-al} below we can write
\[
\al_1 = \frac{p-1}{p} \frac{\ga_1^{p}-1}{\ga_1^{p-1}-1}
\quad \text{and} \quad
\al_2 = \frac{p-1}{p} \frac{\ga_2^{p}-1}{\ga_2^{p}-\ga_2}
\]
in terms of the quotients $\ga_1$ 
and $\ga_2$ associated with $Q_1$ and $Q_2$ as in~\eqref{eq-Q-without-k}
by means of Proposition~\ref{prop-Q-increasing} below. 
A direct calculation then gives
\begin{align*}
\al_1^{p+1/(\al_2-1)} 
     &= \biggl( \frac{p(\ga_1^{p-1}-1)}{(p-1)(\ga_1^p-1)} 
        \biggr)^{\frac{p(p-1)(\ga_2-1)}{\ga_2^p-1-p(\ga_2-1)}}, \\
\al_2^{p+1/(\al_1-1)} 
     &= \biggl( \frac{(p-1)(\ga_2^p-1)}{p(\ga_2^{p}-\ga_2)} 
        \biggr)^{\frac{p(p-1)\ga_1^{p-1}(\ga_1-1)}
                     {p\ga_1^{p-1}(\ga_1-1)-(\ga_1^p-1)}}.
\end{align*}
In particular, for $p=2$ and $Q_1=Q_2=Q$ (and thus $\ga_1=\ga_2=\ga$),
these formulas simplify to
\[
\al_1^{p+1/(\al_2-1)} = \biggl( \frac{2}{\ga+1} \biggr)^{2/(\ga-1)},
\]
which is increasing with respect to $\ga$ and has limit
$1/e$ as $\ga \to 1 \limplus$, while
\[
\al_2^{p+1/(\al_1-1)} = \biggl( \frac{\ga+1}{2\ga} \biggr)^{2\ga/(\ga-1)}
= \biggl(1- \frac{\ga-1}{2\ga} \biggr)^{2\ga/(\ga-1)} < \frac{1}{e}
\]
for all $\ga>1$.
Thus
$\al_2^{p+1/(\al_1-1)} < 1/e < \al_1^{p+1/(\al_2-1)} < 1$
for all $\ga>1$, and hence 
\[
\tQ > Q + (Q-1)/e
\] 
in this case, which is better than the estimate
$\frac{4}{3}Q - \frac{1}{3}= Q+ \frac{1}{3}(Q-1)$ in Example~\ref{ex-1},
but worse than~\eqref{Qt}.

\section{An upper bound for three (or more) functions}
\label{sect-3-or-more}

It is possible to get estimates for the quasisuperminimizing
constant for the minimum of several quasisuperminimizers 
by iteratively using  the estimate for the minimum of two functions. 
The obtained estimate often depends on the order in which the minima are taken. 
This suggests that better estimates could be obtained, if we directly consider the minimum of all of the involved functions and as in the proof of
Theorem~\ref{thm:min2} use all the information 
that is available from the fact that all the functions are quasisuperminimizers with the original constants.

To estimate the quasisuperminimizer constant for the minimum $u$ of $N$ 
quasisuperminimizers $u_i$, let $0 \le \phi\in W^{1,p}_0(\Om)$ be 
arbitrary and set
$v=u + \phi$.
For each $i=1,\ldots,N$ and $S \subset \{1,\ldots,N\}$ with $i\notin S$ let 
\[
u_{S}=\min_{s\in S}\{u_{s},v\}
\quad \text{and} \quad 
\phi_{i,S}=(u_S-u_i)_\limplus.
\]
Then $0\le\phi_{i,S}\le\phi$ and hence $\phi_{i,S} \in W^{1,p}_0(\Om)$.
(Note that $\phi_{i,S}=0$ if $i\in S$.)
Testing~\eqref{eq-qm-Rn} for each $u_i$ with $\phi_{i,S}$ provides us with
$N 2^{N-1}$ inequalities of the form
\begin{equation}   \label{eq-ineqs-for-N-fcions}
\int_{u_i<u_S} |\nabla u_i|^p\,dx\leq Q_{i}\int_{u_i<u_S} |\nabla u_S|^p\,dx.
\end{equation}

This leads to a linear programming problem, 
which is solvable in polynomial time with respect to the number of the conditions. 

\begin{remark}   \label{rem-sharp-ineqs}
When formulating the linear programming problem one
can without loss of generality assume that
the sets $\{x \in \Om : u_i(x)=u_j(x)\}$, $1 \le i < j \le N$,
all have measure zero;
this follows from  the fact that we can 
approximate each $u_i$  from below using $u_i-q_i$, with rational $q_i  \ge 0$,
and the corresponding minima increase to $u$, while preserving the
quasisuperminimizing constant,
by Theorem~6.1 in Kinnunen--Martio~\cite{KiMa03}.
\end{remark}

For example, when $N=3$, we obtain $12$ conditions. We used \emph{Mathematica} 
to solve this linear programming problem and obtained the following result. 
Below we provide a direct proof without relying on \emph{Mathematica}.
However, the \emph{Mathematica} calculation shows that the constant
obtained here is the best possible using only the information
above.

\begin{thm} \label{thm-3Q}
Let $u_{i}$ be a $Q_{i}$-quasisuperminimizer 
 for $i=1,2,3$.
Let 
\[
 P=  2Q_1 Q_2 Q_3 - Q_1 Q_2 -Q_2 Q_3 - Q_3 Q_1 + 1
\]
and, with $\{j,k\}=\{1,2,3\} \setm \{i\}$, 
\[
   R_i = \begin{cases}
        0, & \text{if } Q_j=Q_k=1,  \\
        \displaystyle\frac{(Q_j-1)(Q_k-1)(Q_j-1+Q_k-1)}{Q_j Q_k-1}, &
        \text{otherwise}.
        \end{cases}
\]
 Then $\min\{u_{1},u_{2},u_{3}\}$ is a $\bQ$-quasisuperminimizer with 
\[
\bQ=\frac{Q_1 Q_2 Q_3}{P} (R_1 + R_2 + R_3)
\]
unless 
at least two of the $Q_i$ equal $1$, 
say $Q_2=Q_3=1$, in which case $\bQ=Q_1$.
\end{thm}

It is easily verified that the choice $Q_3=1$  gives
the expression in Theorem~\ref{thm:min2}.
When $Q_1=Q_2=Q_3$, it is also easy to verify that the constant
gets the following simpler form.

\begin{cor}\label{cor-3Q} 
Let $u_{1},u_{2},u_{3}$ be $Q$-quasisuperminimizers. 
Then $\min\{u_{1},u_{2},u_{3}\}$ is a $6Q^{3}/(Q+1)(2Q+1)$-quasisuperminimizer.
\end{cor}

This estimate is slightly better than what we would have obtained 
 by iterating Theorem~\ref{thm:min2}: 
First, the minimum of $u_{1}$ and $u_{2}$ is a $2Q^2/(Q+1)$-quasisuperminimizer, 
and then the minimum of $\min\{u_{1},u_{2}\}$ and $u_{3}$, 
i.e.\ the minimum of a $2Q^2/(Q+\nobreak 1)$- and a $Q$-quasisuperminimizer,  is 
a $2Q^3(3Q+2)/(Q+1)(2Q^2+2Q+1)$-quasi\-super\-minimizer,
by Theorem~\ref{thm:min2}.
However both of these estimates give values close to $3Q$ for large values of $Q$.

We now explain how Theorem~\ref{thm-3Q} can be proved without 
the use of Mathematica. 

\begin{proof}[Proof of Theorem~\ref{thm-3Q}]
If $\min\{Q_1,Q_2,Q_3\}=1$, we have already noticed that the result
follows from Theorem~\ref{thm:min2}, so we assume that $\min\{Q_1,Q_2,Q_3\}>1$.

The proof is similar to the proof of Theorem~\ref{thm:min2}, it just requires
more book keeping.
There are
$12$ inequalities of the form~\eqref{eq-ineqs-for-N-fcions} at our
disposal. 
More precisely, 
for $S=\emptyset$, there are three inequalities
\[
\int_{u_i<v} |\nabla u_i|^p\,dx\leq Q_i \int_{u_i<v} |\nabla v|^p\,dx,
\tag{$E_i$}
\]
$i=1,2,3$.
For singleton $S=\{j\}$, $j\ne i$, we obtain six possible inequalities, namely
\[
\int_{u_i<\min\{u_j,v\}} |\nabla u_i|^p\,dx 
\leq Q_i\int_{u_i<\min\{u_j,v\}} |\nabla \min\{u_j,v\}|^p\,dx,
\tag{$E_{ij}$}
\]
$i,j=1,2,3$, $i\ne j$. 
Finally, for $S=\{j,k\}$, $i\notin S$, we have three inequalities
\[
\int_{u_i<\min\{u_j,u_k,v\}} |\nabla u_i|^p\,dx 
\leq Q_i \int_{u_i<\min\{u_j,u_k,v\}} |\nabla \min\{u_j,u_k,v\}|^p\,dx,
\tag{$\widehat{E}_i$}
\]
$i=1,2,3$.

Depending on the choice of the set $S$ and on the sizes of the functions
$u_1$, $u_2$, $u_3$ and $v$, the sets of integration in these equations
split into three different sets, where also $u_S=\min_{s\in S}\{u_{s},v\}$
equals different $u_i$ or $v$.

Let $\pi=(ijk)$ be a fixed but arbitrary permutation of the set $\{1,2,3\}$.
Then the following subsets of the set 
$A=\{x\in\Om: \phi(x)>0 \text{ and } u_i(x)<u_j(x)<u_k(x)\}$ are of interest:
\begin{align*}
A_0 &= \{x\in A: u_i(x)<u_j(x)<u_k(x)<v(x)\}, \\
A_1 &= \{x\in A: u_i(x)<u_j(x)<v(x)<u_k(x)\}, \\
A_2 &= \{x\in A: u_i(x)<v(x)<u_j(x)<u_k(x)\}.
\end{align*}
(Note that by Remark~\ref{rem-sharp-ineqs}, we can assume that
all the sets $\{x \in \Om : u_i(x)=u_j(x)\}$ have measure zero.)
We shall now check in which of the above inequalities 
these sets appear as parts of the sets of integration.
We shall also keep track of which function then appears in the left-hand side
(LHS) and in the right-hand side (RHS).
It is immediate that none of $A_0$, $A_1$ and $A_2$ is present in the equations
\Eji, \Eki, \Ekj, \Ehj{} or \Ehk.
The set $A_2$ appears only in \Ei, \Eij, \Eik{} and \Ehi,
and the function in the LHS is then always $u_i$, while the one in the RHS is
always $v$.
For the sets $A_0$ and $A_1$, the choices of funtions are more complicated 
and are summarized in Table~\ref{table-sets-of-integration-2}.

\begin{table}[tbh]
\begin{center}
\begin{tabular}{|c|c||c|c||c|} 
\hline
Integral  & Appears in & Gradient & Gradient & Constant \\
over the set & the inequality & in the LHS & in the RHS & in the RHS \\
\hline
\hline
$A_0$ & \Ei & $\grad u_i$ & $\grad v$ & $Q_i$\vphantom{$a_{j^k}^{k^{k^k}}$} \\
\hline
$A_0$ & \Ej & $\grad u_j$ & $\grad v$ & $Q_j$\vphantom{$a_{j^k}^{k^{k^k}}$} \\
\hline
$A_0$ & \Ek & $\grad u_k$ & $\grad v$ & $Q_k$\vphantom{$a_{j^k}^{k^{k^k}}$} \\
\hline
$A_0$ & \Eij & $\grad u_i$ & $\grad u_j$ & $Q_i$\vphantom{$a_{j^k}^{k^{k^k}}$} \\
\hline
$A_0$ & \Eik & $\grad u_i$ & $\grad u_k$ & $Q_i$\vphantom{$a_{j^k}^{k^{k^k}}$} \\
\hline
$A_0$ & \Ejk & $\grad u_j$ & $\grad u_k$ & $Q_j$\vphantom{$a_{j^k}^{k^{k^k}}$} \\
\hline
$A_0$ & \Ehi & $\grad u_i$ & $\grad u_j$ & $Q_i$\vphantom{$a_{j^k}^{k^{k^k}}$} \\
\hline
\hline
$A_1$ & \Ei & $\grad u_i$ & $\grad v$ & $Q_i$\vphantom{$a_{j^k}^{k^{k^k}}$} \\
\hline
$A_1$ & \Ej & $\grad u_j$ & $\grad v$ & $Q_j$\vphantom{$a_{j^k}^{k^{k^k}}$} \\
\hline
$A_1$ & \Eij & $\grad u_i$ & $\grad u_j$ & $Q_i$\vphantom{$a_{j^k}^{k^{k^k}}$} \\
\hline
$A_1$ & \Eik & $\grad u_i$ & $\grad v$ & $Q_i$\vphantom{$a_{j^k}^{k^{k^k}}$} \\
\hline
$A_1$ & \Ejk & $\grad u_j$ & $\grad v$ & $Q_j$\vphantom{$a_{j^k}^{k^{k^k}}$} \\
\hline
$A_1$ & \Ehi & $\grad u_i$ & $\grad u_j$ & $Q_i$\vphantom{$a_{j^k}^{k^{k^k}}$} \\
\hline
\hline
$A_2$ & \Ei & $\grad u_i$ & $\grad v$ & $Q_i$\vphantom{$a_{j^k}^{k^{k^k}}$} \\
\hline
$A_2$ & \Eij & $\grad u_i$ & $\grad v$ & $Q_i$\vphantom{$a_{j^k}^{k^{k^k}}$} \\
\hline
$A_2$ & \Eik & $\grad u_i$ & $\grad v$ & $Q_i$\vphantom{$a_{j^k}^{k^{k^k}}$} \\
\hline
$A_2$ & \Ehi & $\grad u_i$ & $\grad v$ & $Q_i$\vphantom{$a_{j^k}^{k^{k^k}}$} \\
\hline
\end{tabular}
\end{center}
\caption{}
\label{table-sets-of-integration-2}
\end{table}
We multiply the inequalities
\Ei, \Eij{} and \Ehi{} by $x_i$, $x_{ij}$ and $\hat{x}_i$, respectively,
and sum up.
We have $u=u_i$ everywhere in the set $A= A_0\cup A_1 \cup A_2$, and hence to 
show that $u$ is a quasisuperminimizer, we need to keep track of
$\int_A|\grad u_i|^p\,dx$ in the LHSs
and of $\int_A|\grad v|^p\,dx$ in the RHSs.
We also want to choose 
$x_i$, $x_{ij}$ and $\hat{x}_i$
so that 
the integrals of $|\grad u_j|^p$ and 
$|\grad u_k|^p$ in the RHSs are compensated by the same integrals 
in the LHSs.

From Table~\ref{table-sets-of-integration-2}, we see that $\grad u_j$ 
cancels out in $A_0$ and $A_1$ if we have
\begin{equation}\label{cond-u2}
x_j + x_{jk} - Q_i x_{ij} -Q_i \xhi =0,
\end{equation}
and that $\grad u_k$ cancels out in $A_{0}$, if 
\begin{equation}\label{cond-u3}
x_k - Q_i x_{ik} - Q_j x_{jk} = 0.
\end{equation}
In addition, we want the
coefficients in front of the terms containing 
$\grad u_i=\grad u$ in each of the sets $A_0$, $A_1$ and $A_2$ 
to
sum up 
to $1$, i.e.
\begin{equation}\label{cond-u1}
x_i + x_{ij} + x_{ik} + \xhi = 1.
\end{equation}
Considering all permutations of $\{1,2,3\}$ we obtain a linear system
of 12 equations with 12 unknowns. 
However, the system can be simplified, which we do now.
From \eqref{cond-u1} we obtain $\xhi=1-(x_i + x_{ij} + x_{ik})$ and 
inserting this into \eqref{cond-u2} gives
\begin{equation}\label{cond-u2-new}
x_j + x_{jk} + Q_i x_i + Q_i x_{ik} = Q_i.
\end{equation}
From \eqref{cond-u3} we have $Q_i x_{ik} = x_k - Q_j x_{jk}$, which together 
with \eqref{cond-u2-new} leads to 
\begin{align*}
Q_i x_i + x_j + x_k + (1-Q_j)x_{jk} &= Q_i.
\end{align*}
Now, note that this equation is for fixed $i$ symmetric in $j$ and $k$, 
except for the last term in the left-hand side, which thus must be symmetric 
in $j$ and $k$ as well.
Hence, we see that 
\[
(1-Q_j)x_{jk} = (1-Q_k)x_{kj} =: y_i.
\]
Thus 
the above system transforms into the six equations 
\begin{align*}
x_k + S_i y_j + S_j y_i &= 0, \\
Q_i x_i + x_j + x_k + y_i &= Q_i, 
\end{align*}
where $S_i=Q_i/(Q_i-1)$.
It can be written as
\[
\left\{\begin{array}{r}
x  +  Sy  =  0, \\
Rx  +  y  =  c,
\end{array}\right. 
\quad \text{with} \quad
x=\left(\begin{array}{c} 
x_1 \\ x_2 \\ x_3 \end{array}\right)
 \quad \text{and} \quad
y=\left(\begin{array}{c} 
y_1 \\ y_2 \\ y_3 \end{array}\right), \quad
\]
where
\[
S=\left(\begin{array}{ccc} 
0 & S_3 & S_2 \\
S_3 & 0 & S_1 \\
S_2 & S_1 & 0 \end{array}\right), \quad
R=\left(\begin{array}{ccc} 
Q_1 & 1 & 1 \\
1 & Q_2 & 1 \\
1 & 1 & Q_3 \end{array}\right)
 \quad \text{and} \quad
c=\left(\begin{array}{c} 
Q_1 \\ Q_2 \\ Q_3 \end{array}\right).
\]
 From the second equation we have $y=c-Rx$, which transforms the first 
equation into $(SR-I)x=Sc$, whose solution is 
\[
x=(SR-I)^{-1}Sc,
\]
where $I$ stands for the identity matrix.

Now, as we have chosen
 $x$ (and thus $y$), 
so that all extra terms in
the equations \Ei, \Eij, \Ehi, $i,j=1,2,3$, cancel out and the remaining 
ones with $\grad u$ always appear 
with coefficient $1$ in the LHS, 
we need to check how large constants 
appear with $|\nabla v|^{p}$ in the RHS
to determine $\bQ$.
From Table~\ref{table-sets-of-integration-2}, we see that 
$\int_{A_0} |\nabla v|^{p}\,dx$ appears in the right-hand side with a factor
\[
    Q_{A_{0}} = Q_i x_i +Q_j x_j +Q_k x_k.
\]
Similarly, the factors are
\begin{alignat*}{2}
Q_{A_{1}} &= Q_i x_i + Q_j x_j +Q_i x_{ik} +Q_j x_{jk} 
= Q_i x_i + Q_j x_j + x_k \le Q_{A_{0}} &\quad& 
\text{(by \eqref{cond-u3})},\\
Q_{A_{2}} &=Q_i(x_i+x_{ij}+x_{ik}+\xhi) = Q_i &\quad&  
\text{(by \eqref{cond-u1})},
\end{alignat*}
for $\int_{A_1} |\nabla v|^{p}\,dx$ and $\int_{A_2} |\nabla v|^{p}\,dx$,
respectively.
Since the quasiminimizing constant $\bQ$ of $u$ must be at least 
$\max\{Q_1,Q_2,Q_3\}$, we conclude that $Q_{A_{0}}$ is the largest of the
three and 
\begin{equation}  \label{eq-Q-A0}
\bQ = Q_{A_{0}} = c^T x = c^T (SR-I)^{-1}Sc 
= ( ((SR-I)^T)^{-1}c)^T Sc ,
\end{equation}
where ${}^T$ denotes the matrix transpose.
Observe that the value of $Q_{A_0}$ is symmetric in $i$, $j$ and $k$.
An elementary calculation shows that
\[
SR-I=\left(\begin{array}{ccc} 
L_1 & Q_2 L_1 & Q_3 L_1 \\
Q_1 L_2 & L_2 & Q_3 L_2 \\
Q_1 L_3 & Q_2 L_3 & L_3 \end{array}\right), \text{ where } 
L_i = S_j + S_k -1 = \frac{Q_j Q_k -1}{(Q_j-1)(Q_k-1)}, 
\]
for $i\ne j\ne k\ne i$.
Thus, $z:=((SR-I)^T)^{-1}c$ is the unique solution of the system 
$(SR-I)^Tz=c$, which can be equivalently written as
\[
\left(\begin{array}{ccc} 
\frac{L_1}{Q_1} & L_2 & L_3 \\
L_1 & \frac{L_2}{Q_2} & L_3 \\
L_1 & L_2 & \frac{L_3}{Q_3} \end{array}\right)
\left(\begin{array}{c} 
z_1 \\ z_2 \\ z_3 \end{array}\right)
= \left(\begin{array}{c} 
1 \\ 1 \\ 1 \end{array}\right).
\]
Denoting the matrix in the left-hand side by $L$, Cramer's rule gives
\[
z_i = \frac{L_j L_k}{Q_j Q_k \det L} (Q_j-1)(Q_k-1), \quad i\ne j\ne k\ne i,
\] 
where 
\[
\det L = \frac{L_1 L_2 L_3}{Q_1 Q_2 Q_3} (2Q_1 Q_2 Q_3 - Q_1 Q_2 -Q_2 Q_3
- Q_3 Q_1 + 1) =: \frac{L_1 L_2 L_3}{Q_1 Q_2 Q_3} P.
\]
It follows that $z_i = Q_i(Q_j-1)(Q_k-1)/L_i P$.
We also have $Sc=w$, where 
\[
w_i=Q_j S_k+ Q_k S_j = Q_j Q_k \biggl( \frac{1}{Q_j-1} + \frac{1}{Q_k-1} \biggr),
\]
and hence, 
\begin{align*}
w_i z_i &= Q_j Q_k \biggl( \frac{1}{Q_j-1} + \frac{1}{Q_k-1} \biggr)
    \frac{Q_i}{L_i P} (Q_j-1)(Q_k-1) \\
&= \frac{Q_1 Q_2 Q_3}{P} \frac{(Q_j-1)(Q_k-1)(Q_j-1+Q_k-1)}{Q_j Q_k-1}
=: \frac{Q_1 Q_2 Q_3}{P} R_i.
\end{align*}
Consequently, going back
to \eqref{eq-Q-A0}
we obtain
\[
\bQ = Q_{A_{0}} = \sum_{i=1}^3 w_i z_i =  \frac{Q_1 Q_2 Q_3}{P} (R_1 + R_2 + R_3).
\qedhere
\]
\end{proof}

\section{Blowup in pasting lemmas}
\label{sect-pasting}

In this section we shall show that the quasisuperminimizing
constant $Q_1 Q_2$ in the pasting Theorem~\ref{thm-paste-supermin} 
is optimal.
More precisely, we prove the following result.

\begin{thm} \label{thm-Q1Q2-sharp}
Let $p$, $Q_1$ and $Q_2$ be given. 
Then there are $u_1$, $u_2$ and open sets $\Om_1\subset\Om_2$,
such that $u_j$ is a $Q_j$-quasiminimizer in $\Om_j$, $j=1,2$,
and 
\[
     u=\begin{cases}
        u_2, & \text{in } \Om_2 \setm \Om_1, \\ 
        \min\{u_1,u_2\}, & \text{in } \Om_1,
	\end{cases}
\]
is a quasisuperminimizer in $\Om_2$ with the optimal 
quasisuperminimizer constant $Q_1Q_2$.
\end{thm}

This is in sharp contrast to Theorem~\ref{thm:min2}, 
where $\min\{u_1,u_2\}$ is guaranteed to have a quasisuperminimizing constant
$\tQ<Q_1 Q_2$, and moreover, 
\[
Q_1 Q_2-\tQ = Q_1 Q_2 \frac{(Q_1-1)(Q_2-1)}{Q_1 Q_2-1}>0
\quad \text{whenever } Q_1, Q_2>1.
\]

A drawback 
of our proof of Theorem~\ref{thm-Q1Q2-sharp} is that $\Om_1$ is not connected. 
However even when $\Om_1$ is required to be connected we can show,
by varying $p$, the optimality of the
blowup constant in Theorem~\ref{thm-paste-supermin} 
using the following result.

\begin{thm} \label{thm-paste-sharp}
Let $Q_1$, $Q_2$ and $\eps>0$ be given. 
Then there are $p$, $u_1$, $u_2$ and an interval $I=(x_0,1)$, $0\le x_0<1$,
such that $u_1$ is a $Q_1$-quasiminimizer in $I$,
$u_2$ is a $Q_2$-quasiminimizer in $\Om=(0,1)$,
and 
\begin{equation} \label{eq-u}
     u=\begin{cases}
        u_2, & \text{in } \Om \setm I, \\ 
        \min\{u_1,u_2\}, & \text{in } I,
	\end{cases}
\end{equation}
is a $Q$-quasisuperminimizer in $\Om$ with optimal 
quasisuperminimizer constant
\begin{equation} \label{eq-Qpaste-2}
      Q \ge Q_1 Q_2-\eps.
\end{equation}
\end{thm}

\begin{remark} \label{rmk-qsh}
The functions $u_1$, $u_2$ and $u$ 
in the proofs below of Theorems~\ref{thm-Q1Q2-sharp} and~\ref{thm-paste-sharp}
are continuous,  and hence this also demonstrates the sharpness
of the blowup in the pasting lemma for quasisuperharmonic functions 
(Theorem~5.1 in A.~Bj\"orn--Martio~\cite{BMartio}).
\end{remark}

To prove these theorems we need to use some results
on one-corner functions.
In particular, we will use the following result which was obtained by
Uppman~\cite[Section~2.2.3]{uppman}. 
For $p=2$ it is due to Judin~\cite[Example~4.0.25]{judin}.

\begin{thm} \label{thm-ga-p}
Let $0<\alp<\be<\infty$ and $\ga=\be/\al$.
The optimal quasiminimizer constant for
\begin{equation}   \label{eq-one-corner-uppman}
    u(x)=\begin{cases}
        \alp x, & x \le 0, \\
        \be x, & x \ge 0, \\
	\end{cases}
\end{equation}
is
\[
    Q=\frac{(\ga^p+k)(1+k)^{p-1}}{(\ga+k)^p},
\]
where
\[
    k = \frac{p\ga^p(\ga-1)-\ga(\ga^p-1)}{\ga^p-1-p(\ga-1)},
\]
Moreover 
$u$ has the maximal \p-energy allowed by $Q$
on an interval
of the form $[-a,b]$, $a,b >0$,
if and only if $a/b=k$.
\end{thm}

The last part is a consequence of the proof by Uppman 
(or Judin in the case when $p=2$).
Recall from Definition~\ref{deff-determining} that a quasiminimizer
is said to have the maximal \p-energy allowed by $Q$
on an interval $I$ if its \p-energy therein is $Q$-times 
the \p-energy of the linear function with the same boundary values on 
$\bdry I$.
Note also that $k=\ga$ if $p=2$.

We will say that $u$ as in~\eqref{eq-one-corner-uppman} is a \emph{one-corner function} with
\emph{corner} $0$ and \emph{quotient} $\ga$. 
We will mainly be interested in convex one-corner functions as these
are subminimizers and thus $Q$ above is also the optimal 
quasisuperminimizer constant.

\begin{prop}   \label{prop-Q-increasing}
The function $Q(\ga,p)$ is continuous, and moreover
it is strictly increasing with respect to $\ga$.
\end{prop}

\begin{proof}
The continuity follows directly 
from the expressions in Theorem~\ref{thm-ga-p}.

Let $\ga'>\ga$ and let $I=[-a,1]$ be 
an interval such that 
$u$ has the maximal \p-energy allowed by $Q$
on $I$, where $u$ and $Q$ are given by Theorem~\ref{thm-ga-p}
with $\ga=\be/\alp$.
Let $\be'=\ga' \alp > \be$. 
Choose $0<x_0<1$ so that $\alp x_0 + \be' (1-x_0) = \be$
and let
\[
    w(x)=\begin{cases}
        \alp x, & x \le x_0, \\
        \beta' (1-x_0)+\beta, & x \ge x_0. \\
	\end{cases}
\]
Then $w$ is a $Q'=Q(\ga',p)$-quasiminimizer in $I$, $w=u$ 
on $[-a,0]\cup\{1\}$ and $w<u$ in $(0,1)$.
Hence, if $v$ is the linear function
in $I$ with boundary values $v=w$ on $\bdy I$, then
\[
      \int_I |w'|^p \, dx >      \int_I |u'|^p \, dx = Q \int_I |v'|^p \, dx.
\]
This shows that
$Q'>Q$.
\end{proof}

A direct consequence of Proposition~\ref{prop-Q-increasing} is that 
we can  view $\ga$ as a function of $Q$ and $p$, and this
function is strictly increasing with respect to $Q$.
We will also need the following estimate.

\begin{prop} \label{prop-ga>Q}
It is always true that
$Q \le \ga^{p-1}$.
\end{prop}

\begin{proof}
Let $c=k/\ga$.
Then
\[
    Q=\frac{(\ga^p+k)(1+k)^{p-1}}{(\ga+k)^p}
    \le \frac{\ga^p+k}{\ga+k}
    = \frac{\ga^p+c\ga}{(1+c)\ga} 
    \le \frac{\ga^{p}+c\ga^p}{(1+c)\ga}  
    = \ga^{p-1}.
\qedhere
\]
\end{proof}

\begin{remark}
A direct calculation of $\ga^p+k$, $1+k$ and $\ga+k$ yields after
simplifications that
\begin{equation}   \label{eq-Q-without-k}
Q= \frac{(p-1)^{p-1}(\ga^p-1)^p}{p^p(\ga^p-\ga)^{p-1}(\ga-1)}
= \frac{(p-1)^{p-1}}{p^p} \biggl( \frac{\ga^p-1}{\ga^p-\ga} \biggr)^{p-1}
\frac{\ga^p-1}{\ga-1}.
\end{equation}
It is easily verified 
that $\ga^p-1 \ge (\ga-1)\ga^{p-1}$, and 
inserting this into~\eqref{eq-Q-without-k}
gives, together with Proposition~\ref{prop-ga>Q}, the two-sided estimate
\[
Q \le \ga^{p-1} \le \frac{p^{p}Q}{(p-1)^{p-1}}.
\]
\end{remark}

\begin{lem}  \label{lem-ex-one-corner-fn}
Given $\ga>1$, let $Q$ be as in Theorem~\ref{thm-ga-p}.
Then the function
\begin{equation}  \label{eq-def-one-corner}
     u(x)= \begin{cases}
	\alp x, & 0<x \le x_0, \\
	1+\alp\ga (x-1), & x_0 \le x < 1,
	\end{cases}
\end{equation}
with 
\begin{equation}  \label{eq-x0-al}
x_0 =\frac{p\ga^p(\ga-1)-\ga(\ga^p-1)}
	{(p-1)(\ga^p-1)(\ga-1)}
\quad \text{and} \quad
\al = \frac{p-1}{p} \frac{\ga^{p}-1}{\ga^{p}-\ga},
\end{equation}
is the unique one-corner function with the boundary conditions 
$u(0)=0$ and $u(1)=1$ that is convex and
has the maximal \p-energy allowed by $Q$ on $(0,1)$.
\end{lem}

\begin{proof}
For $u$ to be continuous, it is required that 
$\alp x_0 + \alp \ga (1-x_0)=1$, i.e.\ that
\begin{equation} \label{eq-alp-x0}
	\alp= \frac{1}{\ga+x_0(1-\ga)}.
\end{equation}
Theorem~\ref{thm-ga-p} with $a=x_0$ and $b=1-x_0$ gives $x_0/(1-x_0)=k$,
i.e.\ $x_0=k/(k+1)$.
The formula for $k$ from Theorem~\ref{thm-ga-p} then yields after some 
simplification the formula for $x_0$ in~\eqref{eq-x0-al}.
Inserting that into~\eqref{eq-alp-x0} then concludes the proof of the lemma,
since uniqueness follows by construction.
\end{proof}

\begin{remark}
Lemma~\ref{lem-ex-one-corner-fn} can also be proved without an appeal to
Theorem~\ref{thm-ga-p} by maximizing the \p-energy 
\[
	E=\int_0^1 |u_2'|^p \, dx
	   = \alp^p x_0 + \alp^p \ga^p (1-x_0)
         = \frac{x_0+\ga^p(1-x_0)}{(\ga+x_0(1-\ga))^p}
\]
with respect to $x_0$.
\end{remark}

\begin{remark} \label{rmk-Qag}
A straightforward calculation shows that for $Q$ and $\al$ 
from~\eqref{eq-Q-without-k}
and~\eqref{eq-alp-x0} it holds that $Q=Q_\al=Q_{\al\ga}$,
where $Q_\al$ and $Q_{\al\ga}$ are related to $\al$ and $\al\ga$
as in~\eqref{eq-def-Q-al}.
Thus, the optimal one-corner function provided by 
Lemma~\ref{lem-ex-one-corner-fn} is tangent at the end points $1$ and $0$
to the power-like functions $u_Q$ and $\ub_Q$
from~\eqref{eq-def-uQ-ubQ}, respectively.
\end{remark}

The proof below of Theorem~\ref{thm-paste-sharp} is based on varying $p$ 
and the fact that the constant in Theorem~\ref{thm-paste-supermin} 
is independent of $p$.
For fixed $p$ we obtain the following somewhat weaker result.

\begin{prop} \label{prop-paste-trivial-lower-bdd}
Let $p$, $Q_1$ and $Q_2$ be given. 
Then there are $u_1$, $u_2$ and an interval $I=(x_0,1)$, $x_0\ge 0$,
such that $u_1$ is a $Q_1$-quasiminimizer in $I$,
$u_2$ is a $Q_2$-quasiminimizer in $\Om=(0,1)$,
and 
\[
     u=\begin{cases}
        u_2, & \text{in } \Om \setm I, \\ 
        \min\{u_1,u_2\}, & \text{in } I.
	\end{cases}
\]
is a $Q$-quasisuperminimizer in $\Om$ with optimal 
quasisuperminimizer constant
\begin{equation} \label{eq-Qpaste}
      Q \ge Q_1(Q_2-1)+1 = Q_1 Q_2-Q_1+1.
\end{equation}

If moreover $Q_1>1$, then 
the inequality in~\eqref{eq-Qpaste} is strict, i.e.\ 
$Q> Q_1(Q_2-1)+1$.
\end{prop}

\begin{proof}
Using Proposition~\ref{prop-Q-increasing} and Lemma~\ref{lem-ex-one-corner-fn} 
we can find $0 \le x_0<1$, $\ga \ge 1$ and $0<\alp \le 1$ such that
the function $u_2$ given by~\eqref{eq-def-one-corner}
is a $Q_2$-quasiminimizer in $\Om$ 
with the maximal \p-energy allowed by $Q_2$ on $\Om$.
For $Q_2=1$ let $u_2(x)=x$ and $x_0=0$.

Another use of Lemma~\ref{lem-ex-one-corner-fn} provides us with 
a convex one-corner function $u_1$ which
is a $Q_1$-quasiminimizer in $I=(x_0,1)$ 
with 
boundary values $u_1=u_2$ on $\bdy I$ and
maximal \p-energy allowed by $Q_1$ on $I$.
(If $Q_1=1$, we let $u_1\equiv u_2$ on $I$, which is not a one-corner function.)

Since $\al\le1$ and $x_0<1$, we have
\[
     A:=\int_0^{x_0} |u_2'|^p\, dx = \alp^p x_0 < 1.
\]
It then follows that 
\begin{align} \label{eq-A}
     \int_0^1 |u'|^p\,dx 
	&= \int_0^{x_0} |u_2'|^p\, dx + \int_{x_0}^1 |u_1'|^p\, dx
	= A + Q_1\int_{x_0}^1 |u_2'|^p\, dx \\
	&= A + Q_1(Q_2-A) 
	= Q_1Q_2-A(Q_1-1)
	\ge Q_1Q_2-(Q_1-1), \nonumber
\end{align}
where the inequality is strict if $Q_1>1$.
As $v(x)=x$ is the minimizer with boundary values $v=u$ on $\bdy \Om$,
and its \p-energy on $\Om$ is $1$, this concludes the proof.
\end{proof}

We are now ready to prove Theorems~\ref{thm-Q1Q2-sharp} 
and~\ref{thm-paste-sharp}.

\begin{proof}[Proof of Theorem~\ref{thm-Q1Q2-sharp}]
The argument is a modification of the proof of 
Proposition~\ref{prop-paste-trivial-lower-bdd}.
Let $\Om_2=(0,1)\subset\R$ and $u_2$ and $x_0$ be as in the proof of
Proposition~\ref{prop-paste-trivial-lower-bdd}.
Now let $\Om_1=(0,x_0)\cup(x_0,1)$ and choose $u_1$ so that $u_1(x)=u_2(x)$
for $x=0,x_0,1$, and its restrictions
to $(0,x_0)$ and to $(x_0,1)$ are convex one-corner functions provided 
by Lemma~\ref{lem-ex-one-corner-fn}, which are
$Q_1$-quasiminimizers in the respective 
intervals and have the maximal energy therein allowed by $Q_1$. 
But then $u=u_1$ and 
\begin{align*}
\int_0^1 |u'|^p\,dx 
	&= \int_0^{x_0} |u_1'|^p\, dx + \int_{x_0}^1 |u_1'|^p\, dx  \\
	&= Q_1 \int_0^{x_0} |u_2'|^p\, dx + Q_1\int_{x_0}^1 |u_2'|^p\, dx 
	= Q_1 \int_0^1 |u_2'|^p\, dx = Q_1 Q_2.\qedhere
\end{align*}
\end{proof}

\begin{proof}[Proof of Theorem~\ref{thm-paste-sharp}]
We proceed as in the proof of Proposition~\ref{prop-paste-trivial-lower-bdd}.
By Proposition~\ref{prop-ga>Q}, 
$1-\ga^{1-p}\ge 1-1/Q$.
It thus follows from Lemma~\ref{lem-ex-one-corner-fn} that 
\[
      \alp = \frac{p-1}{p} \frac{1-\ga^{-p}}{1-\ga^{1-p}}
           < \frac{p-1}{p} \frac{1}{1-1/Q}
	\to 0, \quad \text{as } p\to1\limplus.
\] 
Hence $A=\alp^p x_0 <\al^p \to 0$, as $p\to1\limplus$, so as in 
\eqref{eq-A},
\[
       \int_0^1 |u'|^p\,dx 
	= Q_1Q_2-A(Q_1-1) \to Q_1Q_2, \quad \text{as } p\to1\limplus.
\]
(Note that $u_1$, $u_2$ and $u$ depend on $p$.)
\end{proof}

\begin{remark} \label{rmk-second}
The estimate~\eqref{eq-Qpaste} in Proposition~\ref{prop-paste-trivial-lower-bdd}
can be replaced by $Q\ge Q_2(Q_1+1)/2$, which gives a better lower bound 
when $Q_2<2$.
Indeed, we always have
\[
\int_0^{x_0} |u_2'|^p\, dx\le \frac{Q_2}{2}
\quad \text{or} \quad 
\int_{x_0}^1 |u_2'|^p\, dx\le \frac{Q_2}{2}.
\]
In the former case, the proof goes through as before, in the latter case, 
replace $u_2$ and $u_1$ by decreasing convex one-corner
functions in $\Om$ and $I$, respectively,
with the maximal \p-energies allowed by 
$Q_2$ and $Q_1$ therein, so that $u_2(0)=1$, $u_1(x_0)=u_2(x_0)$ and 
$u_1(1)=u_2(1)=0$.
In both cases, a direct calculation gives
\begin{equation} \label{eq-Q-second}
Q\ge Q_1 Q_2 -\frac{Q_2}{2}(Q_1-1) = \frac{Q_2(Q_1+1)}{2}.
\end{equation}
\end{remark}

\begin{cor}
Let $p$, $Q_1$, $Q_2$ and open sets
$\Om_1\subsetneq \Om_2=(0,1)$ be given. 
Then there are $u_1$ and $u_2$, which are $Q_1$- and $Q_2$-quasiminimizers
in $\Om_1$ and $\Om_2$, respectively, such that 
\begin{equation}   \label{eq-def-u-paste}
     u=\begin{cases}
        u_2, & \text{in } \Om_2 \setm \Om_1, \\ 
        \min\{u_1,u_2\}, & \text{in } \Om_1,
	\end{cases}
\end{equation}
is a $Q$-quasisuperminimizer in $\Om$ with optimal 
quasisuperminimizer constant satisfying~\eqref{eq-Qpaste}
and \eqref{eq-Q-second}.
\end{cor}

\begin{proof}
Since $\Om_1$ is open, it can be written as a pairwise disjoint union
of open intervals. 
Let $(x_1,x_2)$ be one of them and assume to begin 
with that $x_1>0$.
We can then find $\de>0$ so that 
\[
\Om':=(x_1-k\de,x_1+\de)\subset (0,x_2),
\]
where $k$ is the constant 
associated with $Q_2$ as in Theorem~\ref{thm-ga-p}.

Rescale the functions in
Proposition~\ref{prop-paste-trivial-lower-bdd} 
or Remark~\ref{rmk-second} (depending on which gives a better estimate)
so that they apply to the
sets $\Om'$ and $I':=(x_1,x_1+\de)$ in place of $\Om$ and $I$.
This provides us with one-corner functions $v_1$ and $v_2$, which are 
$Q_1$- and $Q_2$-quasiminimizers in $I'$ and $\Om'$, respectively,
and their pasted function is a $Q$-quasisuperminimizer in $\Om'$ with optimal 
quasisuperminimizer constant $Q$ satisfying~\eqref{eq-Qpaste}
and \eqref{eq-Q-second}.

Now, let $u_2$ be the linear extension of $v_2$ which is a one-corner 
function on the whole of $(0,1)$. 
Also, let 
\[
u_1 = \begin{cases}
        u_2, & \text{in } \Om_1 \setm (x_1,x_2), \\ 
        v_1, & \text{in } (x_1,x_2),
	\end{cases}
\]
where $v_1$ is extended linearly as a one-corner functions on the whole
of $(x_1,x_2)$.
Then the best quasiminimizing constants of $u_1$ and $u_2$ in $\Om_1$ 
and $\Om_2$ are still $Q_1$ and $Q_2$, but their pasted function $u$
given by~\eqref{eq-def-u-paste}
will have its optimal quasisuperminimizing constant 
satisfying~\eqref{eq-Qpaste} and \eqref{eq-Q-second}
in $\Om'$ and thus in $\Om_2$.

If $x_1=0$ then necessarily $x_2<1$ and the above construction can be done for
the interval $(1-x_2,1)$ instead,  replacing $u_1$ and $u_2$
by the decreasing convex one-corner functions $x\mapsto u_1(1-x)$
and $x\mapsto u_2(1-x)$.
\end{proof}

We conclude the paper with further examples of quasiminimizers with explicit
optimal quasiminimizing constants.

\begin{prop}   \label{prop-optimal-zig-zag}
Every strictly increasing continuous piecewise linear function $u$ in $(0,1)$ 
\textup{(}having finitely many corners\/\textup{)}
with alternating slopes $\al$ and $\be$, $\al<\be$, 
is a quasiminimizer in $(0,1)$
with the best quasiminimizing constant $Q$ given by~\eqref{eq-Q-without-k}
with $\ga=\be/\al$.

Moreover, if $u$ has at least one convex {\rm(}concave\/{\rm)} corner, 
then $Q$ is also
the best quasisuperminimizing {\rm(}quasisubminimizing\/{\rm)} constant.
\end{prop}

Clearly, replacing $u$ with $x\mapsto u(1-x)$ gives a strictly decreasing
quasiminimizer with the same best quasiminimizing constant as $u$.
Note also that we do not require that the first segment defining $u$ has 
slope~$\al$, nor that the last segment has slope~$\be$.
However, we do not allow $u$ to be a linear function in the proposition, 
as then $\ga=1$ and 
$Q$ cannot be defined using \eqref{eq-Q-without-k}. Nevertheless,
$Q=1$ is trivially the best quasiminimizing constant 
for $u$ in this case.

\begin{proof}
To show that $Q$ will do, let $0\le a<b\le1$ be arbitrary and consider
the linear function $h$ with $h(a)=u(a)$ and $h(b)=u(b)$.
By splitting $(a,b)$ into several subintervals, whose energies can be
estimated separately, we may assume that 
either 
$h=u$ in $(a,b)$,
$h<u$ in $(a,b)$
or $h>u$ in $(a,b)$. 

If $h>u$ in $(a,b)$, then
moving from $a$ to $b$, we can successively eliminate the concave
corners as follows:
If 
\[
u(x)=\max\{u(x')+\be(x-x'),u(x'')+\al(x-x'')\}
\] 
in the interval $(x',x'')$, 
where $x'$ and $x''$ are two convex corners, then replace $u$ in 
that interval by 
\[
\min\{u(x')+\al(x-x'),u(x'')+\be(x-x'')\}.
\]
This will decrease the number of corners in $(a,b)$ by $2$, while preserving
the \p-energy of $u$ therein.
In the end, this procedure leaves us with a function which in $(a,b)$
coincides with a one-corner function $v$ with slopes $\al$ and $\be$ and
the same \p-energy therein as $u$.
Theorem~\ref{thm-ga-p} shows that $v$ is a $Q$-quasiminimizers in $(a,b)$
and hence 
\[
\int_a^b |h'|^p\,dx \le Q \int_a^b |v'|^p\,dx = Q \int_a^b |u'|^p\,dx.
\]
The argument is similar when $h <u$ in $(a,b)$,
while if $h=u$ in $(a,b)$ we trivially have
$
\int_a^b |h'|^p\,dx = \int_a^b |u'|^p\,dx <   Q\int_a^b |u'|^p\,dx. 
$
As $a$ and $b$ were arbitrary, this shows that $u$ is a $Q$-quasiminimizer.

Finally, if $u$ has at least one convex (concave) corner, then considering
intervals of type $(x_0-k\de,x_0+\de)$, where $x_0$ is one such corner, together
with the last part of Theorem~\ref{thm-ga-p} 
shows that the quasisuperminimizing 
(quasisubminimizing) constant of $u$ cannot be better than $Q$.
As every piecewise linear function with nonequal slopes has at least 
one corner, this concludes the proof.
\end{proof}

\end{document}